\documentclass[11pt,reqno]{article}
\usepackage{amsmath,amsfonts,amsthm}
\usepackage[margin=1in]{geometry}
\usepackage[normalem]{ulem}
\usepackage{ amssymb }

\usepackage[hyphens]{url}
\usepackage{hyperref}
\usepackage{graphicx}
\usepackage{verbatim}
\usepackage{enumitem} 
\usepackage{color}
\usepackage{tikz}
\usetikzlibrary{decorations.fractals}
\usetikzlibrary{patterns,arrows,shapes,decorations.pathreplacing}
\usepackage{bm}

\numberwithin{equation}{section}

\newtheorem{theorem}{Theorem}[section]

\newtheorem{corollary}{Corollary}[section]
\newtheorem{lemma}{Lemma}[section]
\newtheorem{proposition}{Proposition}[section]

\newtheorem{definition}{Definition}[section]
\newtheorem{remark}{Remark}[section]
\newtheorem{example}{Example}[section]

\renewcommand{\P}{\mathbb{P}}

\newcommand{\R}{\mathbb{R}}
\newcommand{\E}{\mathbb{E}}
\newcommand{\cE}{\mathcal{E}}
\newcommand{\N}{\mathbb{N}}
\newcommand{\F}{\mathcal{F}}

\newcommand{\X}{\mathbb{X}}
\newcommand{\bT}{\mathbb{T}}
\newcommand{\B}{\mathcal{B}}

\newcommand{\T}{\mathcal{T}}

\newcommand{\eps}{\varepsilon}

\newcommand{\nada}[1]{}
\definecolor{gb}{rgb}{0, 0.2, 0.8}

\title{The Optimal Equilibrium for Time-Inconsistent Stopping Problems\\-- the Discrete-Time Case\thanks{We would like to thank Jianfeng Zhang for his suggestion which led to the initiation of this project.}}
\author{Yu-Jui Huang\thanks{
University of Colorado, Department of Applied Mathematics, Boulder, CO 80309-0526, USA, email: \texttt{yujui.huang@colorado.edu}. Partially supported by National Science Foundation (DMS-1715439) and the University of Colorado (11003573).}
 \and Zhou Zhou\thanks{
University of Sydney, School of Mathematics and Statistics, NSW 2006, Australia, email: \texttt{zhou.zhou@sydney.edu.au}.}
}
\date{\today}
\begin{document}
\maketitle

\begin{abstract}
We study an infinite-horizon discrete-time optimal stopping problem under non-exponential discounting. A new method, which we call the {\it iterative approach}, is developed to find subgame perfect Nash equilibria. When the discount function induces decreasing impatience, we establish the existence of an equilibrium through fixed-point iterations. Moreover, we show that there exists a unique {\it optimal} equilibrium, which generates larger values than any other equilibrium does {\it at all times}. To the best of our knowledge, this is the first time a dominating subgame perfect Nash equilibrium is shown to exist in the literature of time-inconsistency.
\end{abstract}

\textbf{MSC (2010):} 
49K21, 
60J05,  
91A13, 
93E20. 
\smallskip

\textbf{Keywords:} time-inconsistency, optimal stopping, non-exponential discounting, fixed-point iteration, optimal equilibrium.


\section{Introduction}\label{sec:introduction}
When faced with time-inconsistency in an optimal control or stopping problem, if the agent cannot precommit his future behavior, Strotz \cite{Strotz55} suggested the {\it strategy of consistent planning}. First, one should figure out the strategies that he will actually follow over time. These strategies are formulated as subgame perfect Nash equilibria  in subsequent literature, and constantly called {\it equilibrium strategies}. Next, in Strotz' own words, one should ``choose the best plan among those he will actually follow''. Mathematically, consistent planning gives rise to two problems: 
\begin{itemize}
\item [(a)] 
How do we find equilibrium strategies?
\item [(b)] How do we select the optimal equilibrium strategy? In particular, how do we formulate optimality for equilibrium strategies? 
\end{itemize}
In a discrete-time model with a finite time horizon, backward sequential optimization gives simple answers to (a) and (b). As detailed in Pollak \cite{Pollak68}, given an $N$-period model with time points $t=0,1,...,N$, 
one first determines an optimal strategy at $t=N-1$ for the last period $[N-1,N]$.  
Then, at each $t=N-2, N-3,...,0$, one chooses an optimal strategy for the period $[t,t+1]$, 
in response to his future selves' given strategy on $[t+1,N]$. It is well-known that the resulting strategy on $[0,N]$ is the {\it unique} equilibrium; see e.g. the introductions of \cite{Kocher96} and \cite{Asheim97}, and also Corollary~\ref{coro:finite horizon} below for a similar result.

When the time horizon is infinite, (a) and (b) become much more involved. First, the above backward construction breaks down, as there is no fixed terminal time to start the procedure. While one can still manage to apply the backward idea to infinite-horizon problems, as shown in \cite{PP68}, \cite{PY73}, and \cite{GW07}, among others, the methods employed are normally ad hoc and the main purpose is to demonstrate existence of at least a few specific equilibria. In short, a general systematic approach for finding equilibrium strategies is lacking. More importantly, there can be multiple equilibrium strategies under infinite horizon, as observed in \cite{PP68}, \cite{Asheim97} and \cite{Kocher96}. How to select the appropriate equilibrium is then a genuine problem. The main approach in the literature is to impose additional delicate assumptions, besides subgame perfection, on how an agent makes decisions. 
These assumptions, such as ``reconsideration-proofness'' in 
\cite{Kocher96} and ``revision-proofness'' in 
\cite{Asheim97}, lead to a refined collection of equilibrium strategies, which can be much smaller than the whole set of equilibrium strategies. The hope is that every refined equilibrium {\it might} generate the same value, and thus resolves (b) in a trivial way.  
Despite some positive progress in this direction, refined equilibrium strategies are {\it not} always unique in value. It is still unclear in general how to select the appropriate equilibrium, even though we have a smaller set of candidates.
 
In this paper, we investigate a general infinite-horizon stopping problem in which an agent maximizes his expected payoff, under non-exponential discounting, by choosing an appropriate time to stop a discrete-time 
Markov process $X$ valued in some Polish space $\X$. Under current context, a stopping rule for the agent can be represented by the corresponding stopping region in $\X$. To tackle (a), a new method, which we call the {\it iterative approach}, is developed. Specifically, we formulate {\it equilibrium stopping regions} (or {\it equilibria} for short) as fixed points of an operator $\Theta$, defined precisely in \eqref{Theta} below. To find equilibria, we simply carry out fixed-point iterations: for any stopping region $S\subseteq\X$, we show that
\begin{equation}\label{iteration}
S_\infty := \lim_{n\to\infty}\Theta^n(S)
\end{equation}
is indeed an equilibrium (i.e. $\Theta(S_\infty)=S_\infty$), as long as $\Theta(S)\subseteq S$ holds in the first iteration; see Theorem~\ref{t1} for details.

Note that the traditional backward sequential optimization and our iterative approach both treat an time-inconsistent problem as an intra-personal game between current and future selves. The key difference is that the former uses a {\it sequential-game} point of view, whereas the latter takes up a {\it simultaneous-game} perspective. This new perspective gets rid of structural limitations of a sequential game, which allows the iterative approach to be very flexible: it can be applied with ease to discrete-time infinite-horizon problems (as in the present paper), as well as continuous-time problems (as in Huang and Nguyen-Huu \cite{HN17} and Huang, Nguyen-Huu, and Zhou \cite{HNZ17}); see the detailed discussion below Remark~\ref{rem:finite horizon}.

The iterative approach has novel consequences for (b). We show that there exists a unique equilibrium that {\it always} dominates other equilibria: it generates larger value to the agent than any other equilibrium does, at {\it any} state $x\in\X$ of the process $X$; see Theorems~\ref{t2} and \ref{t3}. This particular equilibrium is certainly ``the best plan among those the agent will actually follow'' as Strotz \cite{Strotz55} prescribed, and thus we call it the {\it optimal} equilibrium. To the best of our knowledge, this is the first time a dominating subgame perfect Nash equilibrium is shown to exist in the literature of time-inconsistency.  

Theorems~\ref{t2} and \ref{t3} hinge on two crucial conditions. First, we require the payoff function of the agent to be either lower or upper semicontinuous, and the transition kernel of the Markov process $X$ to be lower semicontinuous under the weak star topology. Second, the discount function is required to be log sub-additive, i.e. satisfy \eqref{DI} below. This condition corresponds to {\it decreasing impatience}, a widely-observed feature of empirical discounting in Behavioral Economics and Finance; see the discussion below \eqref{DI} for details. Furthermore, we demonstrate in Example~\ref{eg:equili may not exist} that if \eqref{DI} is not satisfied, an equilibrium may not even exist. We can thus consider \eqref{DI} as a necessary condition for all the game-theoretic derivations to resolve time-inconsistency.

It is worth noting that over the past decade, Ekeland and Lazrak \cite{EL06} have aroused vibrant research on time-inconsistency in Mathematical Finance. This includes \cite{EP08}, \cite{EMP12}, \cite{HJZ12}, \cite{Yong12}, \cite{BMZ14}, and \cite{BKM17}, among others, which {\it all} focus on continuous-time models. They work on formulating continuous-time definitions of equilibrium strategies, and characterizing equilibrium strategies through delicate systems of differential equations. 
How to deal with multiple equilibria, nonetheless, has not yet been addressed. It is then of interest to investigate if 
an optimal equilibrium can exist in continuous time, as in our current discrete-time setup. We will take this as a separate research project.



The paper is organized as follows. Section \ref{sec:iterative approach} introduces our time-inconsistent stopping problem, and proposes the iterative approach to finding equilibria. The approach is made rigorous in Section \ref{sec:existence}, where the convergence of the fixed-point iteration \eqref{iteration} is proved. Section \ref{sec:optimal} establishes the existence of a unique optimal equilibrium. 
Section \ref{sec:example} studies a practical stopping problem, which illustrates our theoretic results explicitly. 
Section \ref{sec:conclusion} concludes the paper. 


\section{The Iterative Approach}\label{sec:iterative approach}
Consider a probability space $(\Omega,\F,\P)$ that supports a time-homogeneous Markov process $X=(X_t)_{t=0,1,\dotso}$ taking values in some Polish space $\X$. Let $\mathcal{B}(\X)$ be the family of Borel sets in $\X$, and $Q$ the transition kernel of $X$. 
Specifically, for any $x\in\X$ and $A\in\mathcal{B}(\X)$,
\begin{equation}\label{kernel}
P(X_{t+1}\in A\,|\,X_t=x)=\int_A Q(x,dy),\quad \hbox{for all}\ t=0,1,\dotso.
\end{equation}
Let $\mathbb{F}=(\mathcal{F}_t)_{t\in\N}$ be the filtration generated by $X$, and $\T$ be the collection of all $\mathbb{F}$-stopping times. For each $x\in\X$, if $X_0=x$, we will constantly write $X$ as $X^x$ to emphasize the initial point, and denote by $\E^x[\cdot]$ the expectation conditional on $X_0=x$. 


For any $x\in\X$ and $\tau\in\T$, consider the objective function
\begin{equation}\label{J}
J(x,\tau):=\E^x\left[\delta(\tau)f\left(X_\tau^{}\right)\right].
\end{equation}
Here, $f:\ \X\mapsto\R$ is a payoff (or, profit) function, assumed to be non-negative, bounded, and Borel measurable; $\delta:\ \N\cup\{0\}\mapsto [0,1]$ is a discount function, assumed to be a strictly decreasing with $\delta(0)=1$ and $\lim_{k\to\infty}\delta(k)=0$. For any $\omega\in\Omega$, if $\tau(\omega)=\infty$, we set $\delta(\tau)f\left(X^x_\tau\right)(\omega):=0$ for all $x\in\X$, which is consistent with $\lim_{k\to\infty}\delta(k)=0$ and the boundedness of $f$.

Given current state $x\in\X$, an agent intends to maximize his expected discounted payoff (or, profit) by choosing an appropriate stopping time, i.e.
\begin{equation}\label{v}
\sup_{\tau\in\T} J(x,\tau). 
\end{equation}
Under fairly general conditions, an optimal stopping time for \eqref{v}, denoted by $\widetilde\tau_{x}\in\T$, exists for each $x\in\X$. A natural question is then whether optimal stopping times obtained at different moments, for instance $\widetilde\tau_{x}$ and $\widetilde\tau_{X^{x}_t}$ with $t>0$, are consistent with each other.  

\begin{definition}[Time-Consistency]
The problem \eqref{v} is time-consistent if for any $x\in\X$ and $t\ge 0$, $\widetilde{\tau}_{x}(\omega) = \widetilde{\tau}_{X_t^{x}}(\omega)$ for a.e. $\omega\in \{\widetilde{\tau}_{x}>t\}$. We say \eqref{v} is time-inconsistent if the above relation does not hold.
\end{definition}
It is well-known that the problem \eqref{v} is time-consistent with $\delta(s) := \frac{1}{(1+\beta)^s}$, for any given $\beta>0$, but time-inconsistent in general for an arbitrary discount function $\delta$. Since one may re-evaluate and change his choice of stopping times over time under time-inconsistency, his stopping strategy is not a single stopping time, but a stopping policy defined below.

\begin{definition}\label{def:policy}
A Borel measurable function $\tau:\mathbb{X}\to\{0,1\}$ is called a stopping policy. 
\end{definition}

The intuition behind is as follows. Given current state $x\in\X$, a stopping policy $\tau$ governs when the agent stops: he stops at the first time $\tau(X^{x}_s)$ yields the value 0.

\begin{remark}\label{rem:stopping region}
Definition~\ref{def:policy} can be equivalently formulated using ``stopping regions''.   
By definition, 
\[
\tau:\mathbb{X}\mapsto\{0,1\}\ \hbox{is a stopping policy}\ \iff\ \tau(x)= 1_{S^c}(x)\ \hbox{for some $S\in\B(\X)$}. 
\]
Obviously, $S=\{x\in\X : \tau(x)=0\}$, and we call it the {\it stopping region} of the policy $\tau$. 
\end{remark}

To resolve time-inconsistency, the agent should first figure out ``the strategies that he will actually follow'' as Strotz \cite{Strotz55} proposed. To this end, he needs to take into account his future selves' behavior, and find the best response to that. Specifically, suppose that the agent initially planned to take $S\in\B(\X)$ as the stopping region. Now, at any state $x\in\X$, the agent carries out the game-theoretic reasoning: ``assuming that all my future selves will take $S\in\B(\X)$ as the stopping region, what is the best stopping strategy today in response to that?'' The agent today actually has only two possible actions: stopping and continuation. If he stops, he gets $f(x)$ immediately. If he continues, given that all his future selves will follow $S\in\B(\X)$, he will eventually stop at the moment
\begin{equation}\label{rho}
\rho(x,S):=\inf\{t\geq 1:\ X_t^{x}\in S\}\in\T.
\end{equation}
This leads to the expected payoff $J(x,\rho(x,S))$. Thus, by comparing the two payoffs $f(x)$ and $J(x,\rho(x,S))$, we obtain the best stopping strategy for today, whose stopping region is given by  
\begin{equation}\label{Theta}
\Theta(S):=\{x\in\X:\ f(x)\geq J(x,\rho(x,S))\}.
\end{equation}

\begin{remark}\label{rem:nonempty}
For any $S\in\B(\X)$, $\Theta(S)$ is nonempty. In fact, 
\begin{equation}\label{E}
\Theta(S) \supseteq E:= \{x\in\X : f(x) > M\delta(1)\} \neq\emptyset,\quad \hbox{where}\ \ M:= \sup_{x\in\X} f(x)<\infty.
\end{equation}
Indeed, by \eqref{J}, $J(x,\rho(x,S))\le \delta(1) M < M$. The definition of $M$ ensures that $E$ is nonempty. Now, for any $x\in E$, $f(x)> M \delta(1)\ge J(x,\rho(x,S))$, and thus $x\in \Theta(S)$. Hence, $E\subseteq\Theta(S)$.
\end{remark}

\begin{definition}\label{def:equilibrium}
We call $S\in\mathcal{B}(\X)$ an equilibrium stopping region (or equilibrium for short) if $\Theta(S)=S$. We denote by $\cE$ the set of all equilibria.
\end{definition}

\begin{remark}\label{rem:Theta}
By definition, $S\in\cE$ if and only if
\begin{equation}
\begin{cases}
f(x)\geq J(x,\rho(x,S)),& x\in S,\\
f(x)< J(x,\rho(x,S)),& x\in \X\setminus S.
\end{cases}
\end{equation}
\end{remark}

\begin{remark}\label{rem:equi nonempty}
If an equilibrium exists, it must contain the nonempty set $E$ defined in \eqref{E}. Indeed, if $S\in\B(\X)$ is an equilibrium, then $S=\Theta(S)\supseteq E\neq\emptyset$, where the inclusion follows from Remark~\ref{rem:nonempty}.
\end{remark}



In view of \eqref{Theta}, we can consider $\Theta$ as an operator mapping $\mathcal{B}(\X)$ to itself. An equilibrium is then a fixed point of $\Theta$. To find an equilibrium, we carry out fixed-point iterations: for any initial stopping region $S\in\B(\X)$, we expect $S_\infty$, defined as in \eqref{iteration}, to be an equilibrium. To make this {\it iterative approach} rigorous, we need to show that (i) the limit in \eqref{iteration} converges, so that $S_\infty$ is well-defined, and (ii) $S_\infty$ is indeed an equilibrium, i.e. $\Theta(S_\infty)=S_\infty$. These will be analyzed in detail in Section~\ref{sec:existence} below.


\subsection{Comparison with Backward Sequential Optimization}
In this subsection, we will show that in the finite-horizon case, our iterative approach and the traditional backward sequential optimization yield the same result. Consider a finite horizon $N\in\N$ and let $\bT:=\{0,1,...,{N-1}\}$. The objective function \eqref{J} becomes
\[
J(t,x,\tau):=\E^{t,x}[\delta(\tau-t)f(X_\tau)],\quad \hbox{for}\ (t,x)\in(\bT\cup\{N\})\times\X\ \hbox{and}\ \tau\in\T_t.
\]
Here, $\T_t$ denotes the set of stopping times $\tau\in\T$ with $t\le\tau\le N$ a.s. and $\E^{t,x}$ is the expectation conditional on $X_t=x$. The stopping problem \eqref{v} now takes the form
\[
\sup_{\tau\in\T_t} J(t,x,\tau).
\]
We now formulate stopping policies and equilibrium stopping regions as in Definitions~\ref{def:policy} and \ref{def:equilibrium}.  

\begin{definition}
Under current finite-horizon setting, a Borel measurable function $\tau:\bT\times\X\mapsto\{0,1\}$ is called a stopping policy (here, we use the discrete topology for $\bT$). Moreover, we say $S\in\B(\bT\times\X)$ is an equilibrium stopping region (or equilibrium for short) if $\Theta(S)=S$, where
\begin{equation}\label{Theta'}
\begin{split}
\Theta(S)&:=\{(t,x)\in\bT\times\X:\ f(x)\geq J(t,x,\rho(t,x,S))\},
\end{split}
\end{equation}
with $\rho(t,x,S):= \inf\{s\ge t+1: (s, X^{t,x}_s)\in S\}\wedge N$.
\end{definition}

Since the agent will stop at the terminal time $N$ if he has not stopped before then, the iteration \eqref{iteration} has a fairly simple structure. In the proofs below, for any $S\in\B(\bT\times\X)$, we will use the notation $S_n:= \Theta^n(S)$ for all $n\in\N\cup\{0\}$.

\begin{lemma}\label{lem:finite horizon}
Under current finite-horizon setting, fix $S\in\B(\bT\times\X)$. For any $t\in\bT$, we have
\begin{itemize}
\item [(i)] the $t$-section of $S_n$ coincides with the $t$-section of $S_{N-t}$, for all $n\ge N-t$. That is, 
\[
S_{n}\cap(\{t\}\times\X) = S_{N-t}\cap(\{t\}\times\X)\quad \hbox{for all}\ n\ge N-t.
\]
\item [(ii)] $S_{N-t}\cap(\{t\}\times\X)$ is independent of the choice of $S\in\B(\bT\times\X)$. 
\end{itemize}
\end{lemma}

\begin{proof}
We will prove (i) and (ii) by showing that $S_n\cap(\{t\}\times\X)$ is independent of $n\ge N-t$ and $S\in\B(\bT\times\X)$. For any fixed $t\in\bT$, by \eqref{Theta'} and the definition of $S_n$, we have
\[
S_n\cap(\{t\}\times\X) = \{(t,x): x\in\X,\ f(x)\ge  J(t,x,\rho(t,x,S_{n-1})) \},\quad\forall n\in\N.
\]
For $t=N-1$, since the agent will stop at the terminal time $N$ if he has not stopped before then, we must have $\rho(N-1,x,S')=N$ for any $x\in\X$ and $S'\in\B(\bT\times\X)$. It follows that
\begin{equation}\label{indep. n>=1}
S_n\cap(\{N-1\}\times\X) = \{(N-1,x): x\in\X,\ f(x)\ge  J(N-1,x,N) \},\quad\forall n\in\N.
\end{equation}
This shows that $S_n\cap(\{N-1\}\times\X)$ is independent of $n\ge 1$ and $S\in\B(\bT\times\X)$. 
For $t=N-2$, observe that whether $\rho(N-2,x,S_{n-1})$ is equal to $N-1$ or $N$ depends solely on $S_{n-1}\cap(\{N-1\}\times\X)$. Since \eqref{indep. n>=1} indicates that $S_{n-1}\cap(\{N-1\}\times\X)$ is independent of $n\ge 2$ and $S\in\B(\bT\times\X)$, we conclude that for each $x\in\X$,
\[
T(N-2,x) := \rho(N-2,x,S_{n-1})
\]
is independent of $n\ge 2$ and $S\in\B(\bT\times\X)$. 
It follows that 
\[
S_n\cap(\{N-2\}\times\X) = \{(N-2,x): x\in\X,\ f(x)\ge  J(N-2,x,T(N-2,x)) \}
\]
is independent of $n\ge 2$ and $S\in\B(\bT\times\X)$. A straightforward induction gives the desired result.
\end{proof}

\begin{corollary}\label{coro:finite horizon}
Under current finite-horizon setting, 
\begin{itemize}
\item [(i)] for any $S\in\B(\bT\times\X)$, $S_\infty = S_N$ is an equilibrium; 
\item [(ii)] for any $S, S'\in\B(\bT\times\X)$, we have $S_\infty=S'_\infty$.   
\end{itemize}
Hence, there exists a unique equilibrium.  
\end{corollary}

\begin{proof}
To prove (i), note that Lemma~\ref{lem:finite horizon} (i) immediately implies that for any $n\ge N$, 
\[
S_{n}\cap(\{t\}\times\X) = S_{N}\cap(\{t\}\times\X),\quad \hbox{for all}\ t\in \bT.
\]
That is, $S_n = S_N$ for all $n\ge N$. It follows that $S_\infty=S_N$, and $S_N$ is an equilibrium (as $\Theta(S_N)=S_{N+1}=S_N$). Finally, (ii) is a direct consequence of the independence of $S\in\B(\bT\times\X)$ in Lemma~\ref{lem:finite horizon} (ii). 
\end{proof}

\begin{remark}\label{rem:finite horizon}
The proof of Lemma~\ref{lem:finite horizon} indicates that each additional iteration of $\Theta$ corresponds to a further round of backward optimization in Pollak \cite{Pollak68}. Indeed, the first iteration of $\Theta$ determines the $(N-1)$-section of $S_\infty$, which is precisely the backward optimization in \cite{Pollak68} for the period $[N-1,N]$. Similarly, the $n^{th}$ iteration of $\Theta$, $n=2,3,...,N$, determines the $(N-n)$-section of $S_\infty$, which amounts to backward optimization for the period $[N-n,N-n+1]$. Therefore, under finite horizon, the unique equilibrium $S_\infty$ obtained through the iterative approach is the same as the equilibrium generated by backward sequential optimization. 
\end{remark}

Although giving the same result under finite horizon, backward sequential optimization and our iterative approach are quite different in nature. Let us call the agent at time $t$ ``Player $t$'', for all $t\in\bT$. The standard backward approach views the time-inconsistent problem as a {\it sequential game}. Player $(N-1)$ chooses his action first; Player $(N-2)$ does so later, given Player $(N-1)$'s choice of action; similarly, for all $t< N-2$, Player $t$ makes a decision, given the actions of all his future selves. Note that the sequential game structure entails (i) identifying a specific player who acts first (the ``first player''), and (ii) at most countably many players.

By contrast, our iterative approach materializes the traditional wisdom, ``take into account future selves' actions'', in a {\it simultaneous game}. All players choose their actions simultaneously, without prior knowledge of others' choices. Each player decides his best strategy based on his initial belief on how the others will behave. Expecting that other players are reasoning in the same way, each player pursues recursive reasoning (``I think that you think that I think...''), which is characterized by the fixed-point iteration of $\Theta$ in \eqref{iteration}. Specifically, suppose that an agent initially planned to take $S\in\B(\X)$ as the stopping region. At time $0$, with current state $x_0\in\X$, the agent carries out the game-theoretic thinking as depicted below Remark~\ref{rem:stopping region}, and finds that it is better to take $\Theta(S)$ as the stopping region. Any of his future selves, at a possibly different state $y\in\X$, can reason in the same way and conclude that $\Theta(S)$ is a better choice, too. Thus, this first round of game-theoretic thinking, performed simultaneously by all players (indexed by $x\in\X$), changes the choice of stopping regions from $S$ to $\Theta(S)$. Given this, all players (indexed by $x\in\X$) can carry out the same game-theoretic thinking again simultaneously, and conclude that $\Theta^2(S)$ is a better choice than $\Theta(S)$. This procedure continues until one reaches the equilibrium $S_\infty$, which cannot be improved further by the game-theoretic thinking, i.e. applying $\Theta$.

Note that this gets rid of structural limitations of a sequential game, making the iterative approach very 
flexible: it can be applied easily to (i) infinite discrete time, where the ``first player'' cannot be identified (as in the present paper), and (ii) continuous time, where there are uncountably many players (as in Huang and Nguyen-Huu \cite{HN17} and Huang, Nguyen-Huu, and Zhou \cite{HNZ17}).
 


\section{Existence of an equilibrium}\label{sec:existence}

In this section, we will derive the main convergence result for our iterative approach in Theorem~\ref{t1}, which immediately leads to the existence of an equilibrium in Corollary~\ref{coro:existence of equi}.

First, observe that unlike the finite-horizon case where there is a unique equilibrium (see Corollary~\ref{coro:finite horizon}), an equilibrium may fail to exist under infinite horizon. 

\begin{example}\label{eg:equili may not exist}
Consider $\X=\{1,2\}$, $f(x)=x$, and let
\begin{equation}\label{tran prob}
\mathbb{P}(X_{t+1}=1\,|\,X_t=1)=\mathbb{P}(X_{t+1}=2\,|\,X_t=1)=\frac{1}{2},\quad \P(X_{t+1}=2\,|\,X_t=2)=1
\end{equation}
for all $t=0,1,...$. Set $\delta(1)=\frac{3}{4}$, and we will choose $\delta(k)$, $k=2,3,\dotso$, later.

Suppose there exists an equilibrium $S\in\B(\X)$. Since $J(2,\rho(2,S))=\E^2[\delta(\rho(2,S))X_{\rho(2,S)}] = 2\delta(1)=3/2<2$, by Remark~\ref{rem:Theta} we have $2\in S$. On the other hand, a similar argument shows $1\notin S$. Indeed, if $1\in S$, then
$$J(1,\rho(1,S))=\E^1[\delta(1) X_1]=\frac34\left(\frac12\cdot 1+\frac12\cdot 2\right)=\frac{9}{8}>1,$$
a contradiction to Remark~\ref{rem:Theta}. However, with $1\notin S$, we have
$$J(1,\rho(1,S))=\E^1[\delta(\rho(1,S)) X_{\rho(1,S)}]=2\sum_{k=1}^\infty\left(\frac{1}{2}\right)^k\delta(k)=\frac{3}{4}+2\sum_{k=2}^\infty\left(\frac{1}{2}\right)^k\delta(k).$$
By choosing $\delta(k)$ decaying fast enough as $k$ increases, we can make $J(1,\rho(1,S))<1$, which again contradicts Remark~\ref{rem:Theta}. In short, we have shown that when the discount function $\delta(k)$ decays fast enough, there exists no equilibrium in this example. 
\end{example}

To ensure the existence of an equilibrium, we will assume that the discount function $\delta$ satisfies
\begin{equation}\label{DI}
\delta(i)\delta(j)\leq\delta(i+j),\quad \hbox{for all}\ i,j=0,1,\dotso.
\end{equation}
This condition is closely related to {\it decreasing impatience} ($DI$) in Behavioral Economics. It is well-known in empirical studies (see Remark~\ref{rem:evidence} below) that people admits $DI$: when choosing between two rewards, people are more willing to wait for the larger reward (more patient) when these two rewards are further away in time. For instance, in the two scenarios (i) getting \$100 today or \$110 tomorrow, and (ii) getting \$100 in 100 days or \$110 in 101 days, people tend to choose \$100 in (i), but \$110 in (ii).
Following \cite[Definition 1]{Prelec04} and \cite{Noor2009a}, the discount function $\delta$ induces $DI$ if
\begin{equation}\label{DI def}
\hbox{for any $j=1,2,\dotso$},\quad i\mapsto \frac{\delta(i+j)}{\delta(i)}\ \hbox{is strictly increasing.}
\end{equation}
Observe that \eqref{DI def} readily implies \eqref{DI}, as $\delta(i + j)/\delta(i)\ge \delta(j)/\delta(0) = \delta(j)$ for all $i\ge 0$ and $j\ge 1$. That is, \eqref{DI} is automatically true under $DI$. Note that \eqref{DI} is more general than $DI$, as it obviously includes the classical case $\delta(s) := \frac{1}{(1+\beta)^s}$, for any given $\beta>0$.

\begin{remark}\label{rem:evidence}
There is substantial empirical evidence in Behavioral Economics (see e.g. \cite{Thaler81}, \cite{LT89}, \cite{LP92}) that individuals do not discount exponentially. 
To model empirical discounting, many discount functions have been proposed: in continuous time, hyperbolic in \cite{Ainslie-book-92}, \cite{Prelec04} ($\delta(t) := \frac{1}{1+\beta t}$ with $\beta>0$), generalized hyperbolic in \cite{LP92}, \cite{Laibson97} ($\delta(t) := \frac{1}{(1+\beta t)^k}$ with $\beta, k>0$), and pseudo-exponential in \cite{EL06}, \cite{karp2007non}, \cite{EP08} ($\delta(t) := \lambda e^{-\rho_1 t}+ (1-\lambda) e^{-\rho_2 t}$ with $\lambda\in (0,1)$ and $\rho_1,\rho_2>0$); in discrete time, quasi-hyperboic in \cite{Laibson97} ($\delta(0):=1$, $\delta(i) := \beta \rho^i$, with $\beta$, $\rho\in (0,1)$). Note that all these discount functions induce decreasing impatience, and thus satisfy \eqref{DI}.  

Non-exponential discounting beyond ``decreasing impatience'' has been studied much less. To the best of our knowledge, \cite{Hayek36} is the only work that seriously investigates discounting with ``increasing impatience''. It however suffers the criticism of lacking empirical relevance; see e.g. \cite{NP05}.
\end{remark}

\begin{remark}[Market-implied v.s. innate discounting]
Classically, one can use the price of a zero coupon bond as the discount factor. In this case, \eqref{DI} will imply that the forward rate should never exceed the spot rate, as long as the length of the corresponding time period is the same. This seems to be a strong assumption, as the forward rate is forged by market participants' expectations for future prices and need not have specific relations with the spot rate. 

The discounting discussed above is however {\it market-implied}, distinct from the context of this paper: we focus on an individual's {\it innate} time preference, regardless of other market participants, or whether there is a market to start with. This is the standard setup in Behavioral Economics, and a widespread finding is that individuals discount more intensely in near, than distant, future. Condition \eqref{DI} serves to summarize this feature of innate discounting.
Also, we stress that the payoff we discuss may not be monetary: the function $f$ may describe an agent's utility, level of health or happiness, etc. The connection to zero coupon bonds breaks down for non-monetary payoffs.
\end{remark}

\begin{lemma}\label{l1}
Suppose \eqref{DI} holds. For any nonempty $S\in\mathcal{B}(\X)$, if $\Theta(S)\subseteq S$, then 
\begin{equation}\label{e2}
J(x,\rho(x,S))\leq J\left(x,\rho(x,\Theta(S))\right)\quad\hbox{for all}\ x\in\X.
\end{equation}
In particular, this implies $\Theta^2(S)\subseteq\Theta(S)$. 
\end{lemma}

\begin{proof}
By contradiction, suppose that \eqref{e2} does not hold. Then 
$$\alpha:=\sup_{x\in\X}\left[J(x,\rho(x,S))-J(x,\rho(x,\Theta(S))\right]>0.$$
Choose $y\in\X$ such that
\begin{equation}\label{e3}
J(y,\rho(y,S))-J(y,\rho(y,\Theta(S)))>\delta(1)\alpha.
\end{equation}
Consider the event $A:= A_1\cap A_2$, where
\[
A_1:= \left\{\omega\in\Omega: \rho(y,S)(\omega)<\infty\right\}\quad \hbox{and}\quad A_2:=\left\{\omega\in\Omega : X_{\rho(y,S)}^{y}(\omega)\in S\setminus\Theta(S)\right\}.
\]
On the event $(A_1)^c$, $\Theta(S)\subseteq S$ implies that $\rho(y,\Theta(S))\ge \rho(y,S)=\infty$. In view of the discussion below \eqref{J}, we have $\delta(\rho(y,S))f\left(X_{\rho(y,S)}\right)=\delta(\rho(y,\Theta(S)))f\left(X_{\rho(y,\Theta(S))}\right)=0$. On the event $A_1\cap(A_2)^c$, $\rho(y,S)<\infty$ and $\Theta(S)\subseteq S$ implies that $X_{\rho(y,S)}^{y}\in\Theta(S)$. It follows that $\rho(y,\Theta(S))=\rho(y,S)$, and thus $\delta(\rho(y,S))f\left(X_{\rho(y,S)}\right)=\delta(\rho(y,\Theta(S)))f\left(X_{\rho(y,\Theta(S))}\right)$. As a result, we have
\begin{align}
J&(y,\rho(y,S))-J(y,\rho(y,\Theta(S)))\nonumber\\
&=\E^y\left[1_{A}\ \left(\delta(\rho(y,S))f\left(X_{\rho(y,S)}\right)-\delta(\rho(y,\Theta(S)))f\left(X_{\rho(y,\Theta(S))}\right)\right)\right]\nonumber\\
&=\E^y\left[1_{A}\ \delta(\rho(y,S))\ \E^y\left[f\left(X_{\rho(y,S)}\right)-\frac{\delta(\rho(y,\Theta(S)))}{\delta(\rho(y,S))}f\left(X_{\rho(y,\Theta(S))}\right)\Big|\ \mathcal{F}_{\rho(y,S)}\right]\right]\nonumber\\
&= \E^y\bigg[1_{A}\ \delta(\rho(y,S))\ \bigg(f\left(X_{\rho(y,S)}\right)\nonumber\\
&\hspace{1.5in}-\E^y\bigg[\frac{\delta\left(\rho(y,S)+\rho(X_{\rho(y,S)},\Theta(S))\right)}{\delta(\rho(y,S))}f\left(X_{\rho(y,\Theta(S))}\right)\Big|\ \mathcal{F}_{\rho(y,S)}\bigg]\bigg)\bigg]\nonumber\\
&\le \E^y\left[1_{A}\ \delta(\rho(y,S))\ \left(f\left(X_{\rho(y,S)}\right)-\E^y\left[\delta\left(\rho(X_{\rho(y,S)},\Theta(S))\right)f\left(X_{\rho(y,\Theta(S))}\right)\Big|\ \mathcal{F}_{\rho(y,S)}\right]\right)\right],\label{l1-1}
\end{align}
where the inequality follows from \eqref{DI} and the nonnegativity of $f$. Thanks to the strong Markov property of $X$, it holds a.s. that
\begin{align*}
\E^y&\left[\delta\left(\rho(X_{\rho(y,S)},\Theta(S))\right)f\left(X_{\rho(y,\Theta(S))}\right)\Big|\ \mathcal{F}_{\rho(y,S)}\right] 1_A\\
&= \E^{X_{\rho(y,S)}}\left[\delta\left(\rho(X_{\rho(y,S)},\Theta(S))\right)f\left(X_{\rho(X_{\rho(y,S)},\Theta(S))}\right)\right] 1_A= J\left(X_{\rho(y,S)},\rho(X_{\rho(y,S)}, \Theta(S))\right) 1_A.
\end{align*}
This, together with \eqref{l1-1}, gives
\begin{align*}
J&(y,\rho(y,S))-J(y,\rho(y,\Theta(S)))\nonumber\\
&\le \E^y\left[1_{A}\ \delta(\rho(y,S)) \left\{f\left(X_{\rho(y,S)}\right)- J\left(X_{\rho(y,S)},\rho(X_{\rho(y,S)}, \Theta(S))\right)  \right\}\right].
\end{align*}
From the definition of $\Theta$, we have $f(x)<J(x,\rho(x,S))$ for all $x\notin\Theta(S)$. Thus, from the definition of $A$, the above inequality implies 
\begin{align*}
J&(y,\rho(y,S))-J(y,\rho(y,\Theta(S)))\nonumber\\
&\le \E^y\left[1_{A}\ \delta(\rho(y,S)) \left\{J\left(X_{\rho(y,S)},\rho(X_{\rho(y,S)},S)\right)- J\left(X_{\rho(y,S)},\rho(X_{\rho(y,S)}, \Theta(S))\right)  \right\}\right]\\
&\le \E^y\left[1_{A}\ \delta(\rho(y,S)) \alpha \right] \le \delta(1)\alpha. 
\end{align*}
This however contradicts \eqref{e3}. Hence, \eqref{e2} is established.

By Remark~\ref{rem:nonempty}, $\Theta(S)\neq\emptyset$. 
For any $x\notin\Theta(S)$, the definition of $\Theta$ and \eqref{e2} yield $f(x)<J(x,\rho(x,S))\leq J\left(x,\rho(x,\Theta(S))\right)$, which implies $x\notin\Theta^2(S)$. It follows that $\Theta^2(S)\subseteq\Theta(S)$.
\end{proof}

We are ready to present the main convergence result for our iterative approach.

\begin{theorem}\label{t1}
Suppose \eqref{DI} holds. For any nonempty $S\in\mathcal{B}(\X)$ such that $\Theta(S)\subseteq S$, 
\begin{equation}\label{Theta^infty}
S_\infty:=\bigcap_{n=1}^\infty\Theta^n (S)\neq\emptyset
\end{equation}
is an equilibrium.
\end{theorem}

\begin{proof}
Note that $S\in\mathcal{B}(\X)$ implies $\Theta(S)\in\mathcal{B}(\X)$, which in turn yields $\Theta^n(S)\in\mathcal{B}(\X)$ for all $n\in\N$. Therefore, $S_\infty\in\mathcal{B}(\X)$ by definition. We claim that for all $x\in\X$ and $\omega\in\Omega$,
\begin{equation}\label{e4}
\rho(x,S_\infty)(\omega)=\lim_{n\rightarrow\infty}\rho(x,\Theta^n(S))(\omega).
\end{equation}
Fix $x\in\X$ and $\omega\in\Omega$. By Lemma \ref{l1}, $\{\Theta^n(S)\}_{n\in\N}$ is a nonincreasing sequence of nonempty Borel sets, and thus $\rho(x,\Theta^n(S))(\omega)$ is by definition nondecreasing in $n$. The limit on the right hand side of \eqref{e4} is therefore well-defined, and
$$\rho(x,S_\infty)(\omega)\geq\lim_{n\rightarrow\infty}\rho(x,\Theta^n(S))(\omega)=:\beta(\omega).$$
If $\beta(\omega)=\infty$, then \eqref{e4} trivially holds. If $\beta(\omega)\in\N$, then there exists $N\in\N$ large enough such that $\rho(x,\Theta^n(S))(\omega)=\beta(\omega)$ for all $n\ge N$. This implies, by \eqref{rho}, that $X_\beta^{x}(\omega)\in\Theta^n(S)$ for all $n\ge N$. Hence, $X_\beta^{x}(\omega)\in\bigcap_{n=1}^\infty\Theta^n(S)=S_\infty$. It follows that $\rho(x,S_\infty)(\omega)\leq\beta(\omega)$. We thus conclude that \eqref{e4} holds.

By Remark~\ref{rem:nonempty}, $\Theta^n(S) \supseteq E$ for all $n\in\N$, with $E\neq\emptyset$ defined in \eqref{E}. Thus, $S_\infty\supseteq E$ by definition, and is thus nonempty. 
For any $x\in S_\infty$, since $x\in\Theta^n(S)$ for all $n\in\N$, we have $f(x)\geq J\left(x,\rho(x,\Theta^{n-1}(S))\right)$ for all $n\in\N$. As $n\to\infty$, 
\begin{align*}
f(x)&\geq\lim_{n\rightarrow\infty} J\left(x,\rho(x,\Theta^{n-1}(S))\right)\\
&\hspace{-0in}=\lim_{n\rightarrow\infty} \E^x\left[\delta\left(\rho(x,\Theta^{n-1}(S))\right) f\left(X_{\rho(x,\Theta^{n-1}(S))}\right)\right]\\
&\hspace{-0in}= \E^x\left[\delta\left(\rho(x,S_{\infty})\right) f\left(X_{\rho(x,S_{\infty})}\right)\right]
=J\left(x,\rho(x,S_\infty)\right)
\end{align*}
where the third line follows from the dominated convergence theorem and \eqref{e4}. Thus, $x\in\Theta(S_\infty)$, and we conclude that $S_\infty\subseteq \Theta(S_\infty)$. On the other hand, for $x\notin S_\infty$, there exists $N\in\N$ such that $x\notin\Theta^n(S)$ for all $n\geq N$. This implies $f(x)<J\left(x,\rho(x,\Theta^{n-1}(S))\right)$ for all $n\geq N$. It follows that
\begin{equation*}
f(x)<J(x,\rho(x,\Theta^{N-1}(S))\leq\lim_{n\rightarrow\infty}J\left(x,\rho(x,\Theta^{n-1}(S))\right)=J\left(x,\rho(x,S_\infty)\right),
\end{equation*}
where the second inequality is due to \eqref{e2} and the last equality follows again from the dominated convergence theorem and \eqref{e4}.
This implies $x\notin \Theta(S_\infty)$, and we thus conclude that $\Theta(S_\infty)\subseteq S_\infty$. Hence, we have obtained $\Theta(S_\infty)=S_\infty$, i.e. $S_\infty$ is an equilibrium.  
\end{proof}

\begin{remark}
The iterative approach embodies the hierarchy of strategic reasoning in Game Theory, introduced in \cite{Stahl93} and \cite{SW94}. For any 
$S\in\B(\X)$, each application of $\Theta$ to $S$ represents an additional level of strategic reasoning. Specifically, $\Theta^n(S)$ corresponds to level-$n$ strategic reasoning in \cite{SW94}, and $S_\infty = \bigcap_{n\in\N}\Theta^n(S)$ reflects full rationality of ``smart$_\infty$'' players in \cite{Stahl93}.
\end{remark}

Theorem~\ref{t1} immediately gives the existence of at least one equilibrium. 

\begin{corollary}\label{coro:existence of equi}
Suppose \eqref{DI} holds. Then, $\X_\infty$, defined as in \eqref{Theta^infty}, is an equilibrium. 
\end{corollary}

\begin{proof}
Since $\X$ trivially satisfies $\Theta(\X)\subseteq \X$, the result follows from Theorem~\ref{t1} directly. 
\end{proof}

In general, the uniqueness of equilibria does not hold, as the next example demonstrates. 

\begin{example}\label{eg:multiple equilibriums}
Consider $\X=\{1,2\}$, $f(x)=x$, and let the transition probabilities be given as in \eqref{tran prob}. For any $\eps>0$, take
$$\delta(k) := \frac{7}{12}(1-\eps)^k,\quad k=1,2,\dotso.$$
It can be easily verified that \eqref{DI} is satisfied. When $\eps>0$ is small enough, we claim that $S:=\{2\}$ and $\X=\{1,2\}$ are both equilibria. Indeed, by similar calculations in Example~\ref{eg:equili may not exist},
\begin{align*}
J(1,\rho(1,S))&=\E^1[\delta(\rho(1,S)) X_{\rho(1,S)}]=2\sum_{k=1}^\infty\left(\frac{1}{2}\right)^k\delta(k)=\frac{7}{6}\sum_{k=1}^\infty\left(\frac{1-\eps}{2}\right)^k,\\
J(2,\rho(2,S))&=\E^2[\delta(\rho(2,S))X_{\rho(2,S)}] = 2\delta(1)=\frac76 (1-\eps)<2.
\end{align*}
As $\eps\downarrow 0$, $J(1,\rho(1,S))$ approaches $7/6$ and thus strictly larger than $1$. Hence, by Remark~\ref{rem:Theta}, $S=\{2\}$ is an equilibrium as $\eps>0$ small enough. On the other hand,
\begin{align*}
J(1,\rho(1,\X))&=\delta(1)\E^1[X_{\rho(1,\X)}]= \frac78(1-\eps)<1,\quad J(2,\rho(2,\X))=\delta(1) 2 = \frac76 (1-\eps)<2.
\end{align*} 
Thus, by Remark~\ref{rem:Theta}, $\X=\{1,2\}$ is an equilibrium for any $\eps>0$.
\end{example}



%


\section{The Optimal Equilibrium}\label{sec:optimal}

Since there may be multiple equilibria under infinite horizon (see Example~\ref{eg:multiple equilibriums}), how to select the appropriate one is a crucial, unsettled question. In this section, we will first define {\it optimality} of an equilibrium: an equilibrium is optimal if it generates larger value than any other equilibrium does {\it at all times}. While this seems to be a rather strong condition, we show that a unique optimal equilibrium does exist in Theorems~\ref{t2} and \ref{t3}, under appropriate continuity assumptions. 

For any $S\in\cE$, the associated value function is defined by
\[
V(x,S):=f(x)\vee J(x,\rho(x,S))\quad \hbox{for all}\ x\in\X.
\]
By Remark~\ref{rem:Theta}, we immediately have
\begin{equation}\label{V1}
V(x,S) = 
\begin{cases}
f(x),\quad &\hbox{if}\ x\in S,\\
J\left(x,\rho(x,S)\right),\quad &\hbox{if}\ x\in \X\setminus S.
\end{cases}
\end{equation}
This in turn implies that $V$ can be expressed as
\begin{equation}\label{V2}
V(x,S)=J(x,\rho^*(x,S)),\quad \hbox{with}\ \ \rho^*(x,S):=\inf\{t\geq 0:\ X_t^{x}\in S\}.
\end{equation}

\begin{definition}[Optimal Equilibrium]
We say $S^*\in\cE$ is an optimal equilibrium if for any $S\in\cE$,
$$V(x,S^*)\geq V(x,S)\quad\forall\,x\in\X.$$
\end{definition}

The next result shows that if an optimal equilibrium exists, it must be unique.

\begin{proposition}\label{prop:equi unique}
If $S^*\in\cE$ is an optimal equilibrium, then
$
S^* = \bigcap_{S\in\cE} S. 
$
\end{proposition}

\begin{proof}
We only need to show $S^*\subseteq \bigcap_{S\in\cE} S$, as the other inclusion is trivial. By contradiction, suppose there exists $S\in\cE$ such that $S^*\not\subseteq S$. For any $x\in S^*\setminus S$, we deduce from \eqref{V1} and Remark~\ref{rem:Theta} that
$$V(x,S^*)=f(x)<J(x,\rho(x,S))=V(x,S),$$
which contradicts the optimality of $S^*$. 
\end{proof}

\begin{lemma}\label{l2}
Suppose \eqref{DI} holds. For any $S$, $T\in\cE$, 
\begin{equation}\label{value increased}
J\left(x,\rho(x,S\cap T)\right)\geq J(x,\rho(x,S))\vee J(x,\rho(x,T))\ \quad\forall\,x\in\X.
\end{equation}
In particular, this implies $\Theta(S\cap T)\subset S\cap T$.
\end{lemma}

\begin{remark}\label{r1}
This lemma provides a way to construct a ``better'' equilibrium from old equilibria. For any $S$, $T\in\cE$, 
Remark~\ref{rem:equi nonempty} and Lemma~\ref{l2} imply that $S\cap T\neq\emptyset$ and $\Theta(S\cap T)\subseteq S\cap T$. Thus, we can apply Theorem~\ref{t1} to get a new equilibrium $(S\cap T)_\infty$ through the iterative approach. This new equilibrium is 
better than $S$ and $T$ in the sense that
$$V(x,(S\cap T)_\infty)\geq V(x,S)\vee V(x,T)\quad \forall x\in\X,$$
which is a consequence of \eqref{value increased} and \eqref{e2}.
\end{remark}

\begin{proof}[Proof of Lemma \ref{l2}]
Fix $x\in\X$. Throughout this proof, for simplicity of notation, we define, for all $n=0,1,\dotso$, the following:
\begin{itemize}
\item [1.] $y_0:=x$,\quad $y_{2n+1}:=X_{\rho(y_{2n},T)}^{y_{2n}}$,\quad $y_{2n+2}:=X_{\rho(y_{2n+1},S)}^{y_{2n+1}}$;
\item [2.] $\tau_0:=0$,\quad $\tau_{2n+1}:=\rho(y_{2n},T)$,\quad $\tau_{2n+2}:=\rho(y_{2n+1},S)$;
\item [3.] $A_n:=\{\omega\in\Omega :\ \tau_n(\omega)<\infty\ \hbox{and}\ y_n(\omega)\notin S\cap T\}$;
\item [4.] $\mathfrak{E}_n[Z]:=\E^{y_{n-1}}\left[1_{A_n}\delta(\tau_n)\,Z\,\right]$, for any random variable $Z:\Omega\mapsto\R$; 
\item [5.] $J(x,S') := J(x,\rho(x,S'))$, for any $S'\in\B(\X)$. 
\end{itemize}
By the definition of $A_1$, we have 
\begin{align}
\notag J(x,S\cap T)-J(x,T) &=\E^{y_0}\left[1_{A_1}\left(\delta(\rho(y_0,S\cap T))f\left(X_{\rho(y_0,S\cap T)}\right)-\delta(\tau_1)f\left(y_1\right)\right)\right]\\
\notag &=\E^{y_0}\left[1_{A_1}\delta(\tau_1)\E^{y_0}\left[\left(\frac{\delta(\rho(y_0,S\cap T))}{\delta(\tau_1)}f\left(X_{\rho(y_0,S\cap T)}\right)-f\left(y_1\right)\right)\Big|\ \mathcal{F}_{\tau_1}\right]\right]\\
&\geq \E^{y_0}\left[1_{A_1}\delta(\tau_1)\E^{y_0}\left[\left(\delta(\rho(y_1,S\cap T))f\left(X_{\rho(y_0,S\cap T)}\right)-f\left(y_1\right)\right)\Big|\ \mathcal{F}_{\tau_1}\right]\right],\label{J-J}
\end{align}
where the inequality follows from \eqref{DI} and $f$ being nonnegative. The strong Markov property of $X$ implies that
\begin{align}
&\E^{y_0}\left[\delta(\rho(y_1,S\cap T))f\left(X_{\rho(y_0,S\cap T)}\right)\Big|\ \mathcal{F}_{\tau_1}\right] 1_{A_1} \nonumber\\
&=\E^{y_1}\left[\delta(\rho(y_1,S\cap T))f\left(X_{\rho(y_1,S\cap T)}\right)\right] 1_{A_1} = J(y_1,S\cap T) 1_{A_1}.\label{Mark prop}
\end{align}
Moreover, note that on the event $A_1$, we have $y_1\notin S$. Then Remark~\ref{rem:Theta} implies that 
\begin{equation}\label{f<J}
f(y_1) < J(y_1,S).
\end{equation}
With \eqref{Mark prop} and \eqref{f<J}, \eqref{J-J} yields
\begin{align}
J(x,S\cap T)-J(x,T) &\geq \E^{y_0}\left[1_{A_1}\delta(\tau_1)\left(J(y_1,S\cap T)-J(y_1,S)\right)\right]\notag\\
&=\mathfrak{E}_1[J(y_1,S\cap T)-J(y_1,S)].\label{initial}
\end{align}

In the following, we will carry out the above argument recursively. First, repeating the argument above \eqref{initial} for $J(y_1,S\cap T)-J(y_1,S)$, instead of $J(x,S\cap T)-J(x,T)$, yields 
\[
J(y_1,S\cap T)-J(y_1,S) \geq \mathfrak{E}_2[J(y_2,S\cap T)-J(y_2,T)].
\]
This, together with \eqref{initial}, implies 
\begin{equation}\label{initial1}
J(x,S\cap T)-J(x,T) \geq \mathfrak{E}_1\circ\mathfrak{E}_2[J(y_2,S\cap T)-J(y_2,T)].
\end{equation}
Next, repeating the argument above \eqref{initial} for $J(y_2,S\cap T)-J(y_2,T)$, instead of $J(x,S\cap T)-J(x,T)$, and using \eqref{initial1}, we obtain   
\[
J(x,S\cap T)-J(x,T) \geq \mathfrak{E}_1\circ\mathfrak{E}_2\circ \mathfrak{E}_3[J(y_3,S\cap T)-J(y_3,S)].
\]
Continuing this procedure, we get
\begin{equation}\label{for each n}
J(x,S\cap T)-J(x,T) \geq \mathfrak{E}_1\circ\dotso\circ\mathfrak{E}_{2n}[J(y_{2n},S\cap T)-J(y_{2n},T)],\quad \forall n\in\N.
\end{equation}
Since $f$ is bounded, there exists $C>0$ such that $|J(y,S')|\le C$ for all $y\in\X$ and $S'\in\B(\X)$. Then,
\begin{align*}
&\left|\mathfrak{E}_1\circ\dotso\circ\mathfrak{E}_{2n}[J(y_{2n},S\cap T)-J(y_{2n},T)]\right|\\
&\hspace{0.2in}\leq\mathfrak{E}_1\circ\dotso\circ\mathfrak{E}_{2n}[|J(y_{2n},S\cap T)-J(y_{2n},T)|]\leq \delta(1)^{2n}\cdot 2C\rightarrow 0\quad \hbox{as}\ n\to\infty,
\end{align*}
where the convergence follows from $\delta(1)<1$. We then conclude from \eqref{for each n} that
$$J(x,S\cap T)-J(x,T)\geq 0,\quad\forall\,x\in\X.$$
By simply switching the roles of $S$ and $T$ in the proof above, we can also obtain
$$J(x,S\cap T)-J(x,S)\geq 0,\quad\forall\,x\in\X.$$
It follows that \eqref{value increased} is established.

By Remark~\ref{rem:equi nonempty}, $S\cap T\neq\emptyset$. For any $x\notin S\cap T$, if $x\notin S$, then $f(x)<J(x,S)\leq J(x,S\cap T)$, which implies $x\notin\Theta(S\cap T)$; if $x\notin T$, the same argument shows $x\notin\Theta(S\cap T)$. Thus, we conclude that $\Theta(S\cap T)\subseteq S\cap T$.
\end{proof}

Under \eqref{DI}, we have a partial converse of Proposition~\ref{prop:equi unique}.

\begin{proposition}\label{p1}
Suppose \eqref{DI} holds. If $S^*:=\bigcap_{S\in\cE} S$ is an equilibrium, then it is optimal.
\end{proposition}

\begin{proof}
Let $T$ be an arbitrary equilibrium. By Lemma \ref{l2},  
$$V(x,T)=f(x)\vee J(x,\rho(x,T))\leq f(x)\vee J(x,\rho(x,S^*\cap T))=f(x)\vee J(x,\rho(x,S^*))=V(x,S^*),$$
for all $x\in\X$. Hence, $S^*$ is optimal. 
\end{proof}

The next result establishes the existence of the optimal equilibrium. Under appropriate continuity conditions, it is shown that $S^*:=\bigcap_{S\in\cE} S$, the candidate from  Propositions~\ref{prop:equi unique} and \ref{p1}, is indeed the optimal equilibrium.


\begin{theorem}\label{t2}
Suppose that \eqref{DI} holds and $f$ is upper semicontinuous. In addition, assume that the transition kernel $Q$ in \eqref{kernel} is lower semicontinuous under the weak star topology, i.e. for any bounded Borel measurable $g:\X\mapsto\R$ and $\{x_n\}_{n\in\N}$ in $\X$ with $x_n\rightarrow x$,
\begin{equation}\label{e7}
\liminf_{n\rightarrow\infty}\int_\X g(y)\,Q(x_n,dy)\ge \int_\X g(y)\,Q(x,dy).
\end{equation}
Then, $S^*:=\bigcap_{S\in\cE}S$ is the optimal equilibrium.
\end{theorem}

\begin{proof} 
For any $S\in\mathcal{B}(\X)$, observe that $J(x, \rho(x,S))$ can be expressed as
\[
J(x,\rho(x,S))=\int_\X I(y,S)\,Q(x,dy),
\]
where
\[
I(y,S) := \E^y[\delta(\rho^*(y,S)+1)X_{\rho^*(y,S)}]
\]
and $\rho^*(y,S)$ is defined as in \eqref{V2}. Thus, under \eqref{e7}, $J(x,\rho(x,S))$ is lower semicontinuous in $x$. Now, if $S$ is additionally an equilibrium, then
$$S=\Theta(S)=\{x\in\X:\ f(x)\geq J(x,\rho(x,S))\}$$
is a closed subset of $\X$, thanks to the upper semicontinuity of $f$ and the lower semicontinuity of $J(\cdot,\rho(\cdot,S))$. Hence, $S^*:=\bigcap_{S\in\cE}S$ is also closed and thus Borel measurable. 

Since $S$ is closed and thus $1_S$ is upper semicontinuous for all $S\in\cE$, Proposition 4.1 in \cite{BS12} asserts the existence of a countable subset $(S_n)_{n\in\mathbb{N}}$ of $\cE$ such that
\begin{equation}\label{e8}
1_{S^*}=\inf_{S\in\cE}1_S=\inf_{n\in\N}1_{S_n}.
\end{equation}
It follows that $S^*=\bigcap_{n\in\N}S_n$. Now, let $T_1:=S_1$. By the same argument in 
Remark \ref{r1}, $T_2:=(T_1\cap S_2)_\infty\subseteq T_1\cap S_2$ is an equilibrium with $J(x,\rho(x,T_2))\geq J(x,\rho(x,T_1))$ for all $x\in\X$. Similarly, $T_3:=(T_2\cap S_3)_\infty\subseteq T_2\cap S_3$ is an equilibrium with $J(x,\rho(x,T_3))\geq J(x,\rho(x,,T_2))$ for all $x\in\X$. Repeating this procedure, we get $(T_n)_{n\in\N}$ in $\cE$ such that $T_{n+1}\subseteq T_n\cap S_{n+1}$ and $J(x,\rho(x,T_{n+1}))\geq J(x,\rho(x,T_n))$ for all $n\in\N$. As a result,
$$S^*=\bigcap_{S\in\cE}S\subseteq\bigcap_{n\in\N}T_n\subseteq\bigcap_{n\in\N}S_n=S^*,$$
which implies
$S^*=\bigcap_{n\in\N}T_n$. 
Then, by following the same arguments in the proof of Theorem \ref{t1}, with $\Theta^n(S)$ replaced by $T_n$, we can show that $S^*$ is an equilibrium. It is thus the optimal equilibrium thanks to Proposition~\ref{p1}. 
\end{proof}

When $f$ is instead lower semicontinuous, the existence of the optimal equilibrium still holds.

\begin{theorem}\label{t3}
Suppose that \eqref{DI} holds, $f$ is lower semicontinuous, and the transition kernel $Q$ in \eqref{kernel} is lower semicontinuous under the weak star topology, as stated in \eqref{e7}. Then, $S^*:=\bigcap_{S\in\cE}S$ is the optimal equilibrium.
\end{theorem}

\begin{proof}
As shown in the proof of Theorem~\ref{t2}, \eqref{e7} implies that $J(x,\rho(x,S))$ is lower semicontinuous in $x$ for any $S\in\B(\X)$. Thus, $V(x,S)=f(x)\vee J(x,\rho(x,S))$ is lower semicontinuous in $x$ for any $S\in\cE$. By Proposition 4.1 in \cite{BS12}, there exist $(S_n)_{n\in\N}$ in $\cE$ such that 
\begin{equation}\label{sup V}
\sup_{S\in\cE} V(x,S) = \sup_{n\in\N} V(x,S_n),\quad \forall x\in\X.  
\end{equation}
Let $T_0:=\X$ and $T_1:=S_1$. By the same argument in Remark \ref{r1}, $T_2:=(T_1\cap S_2)_\infty\subseteq T_1\cap S_2$ is an equilibrium with $J(x,\rho(x,T_2))\geq J(x,\rho(x,T_1))$ for all $x\in\X$. Similarly, $T_3:=(T_2\cap S_3)_\infty\subseteq T_2\cap S_3$ is an equilibrium with $J(x,\rho(x,T_3))\geq J(x,\rho(x,,T_2))$ for all $x\in\X$. Repeating this procedure, we get $(T_n)_{n\in\N}$ in $\cE$ such that $T_{n}\subseteq T_{n-1}\cap S_{n}$ and $J(x,\rho(x,T_{n+1}))\geq J(x,\rho(x,T_n))$ for all $n\in\N$. Now, by following the same arguments in the proof of Theorem \ref{t1}, with $\Theta^n(S)$ replaced by $T_n$, we can show that $T^*:= \bigcap_{n\in\N} T_n$ is an equilibrium and $J(x,\rho(x,T^*))\ge J(x,\rho(x,T_{n}))$ for all $x\in\X$ and $n\in\N$. Recalling $T_{n}\subseteq T_{n-1}\cap S_{n}$ and \eqref{value increased}, for each $x\in\X$ we have
\[
J(x,\rho(x,T^*))\ge J(x,\rho(x,T_{n})) = J(x,\rho(x,T_{n}\cap S_{n}))\ge J(x,\rho(x,S_{n})),\quad \forall n\in\N.
\]
This implies $V(x,T^*)\ge V(x,S_n)$ for all $x\in\X$ and $n\in\N$. We therefore conclude from \eqref{sup V} that
\[
\sup_{S\in\cE} V(x,S) = V(x,T^*),\quad \forall x\in\X.  
\]
That is, $T^*$ is an optimal equilibrium. By Proposition~\ref{prop:equi unique}, $T^*= \bigcap_{S\in\cE} S=S^*$.
\end{proof}

\begin{remark}
If $\X$ is a finite or countable set, both the semicontinuity of $f$ and \eqref{e7} are satisfied trivially, under the discrete topology of $\X$.
The discrete topology, however, cannot be used anymore when $\X$ is uncountable. This is because the metric space induced by the discrete topology is not separable when $\X$ is uncountable, and thus forbids the use of Proposition 4.1 in \cite{BS12}, a key step in the proof of Theorem~\ref{t2}. 
\end{remark}

We provide below a sufficient condition for \eqref{e7}. 

\begin{remark}
Assume that, for any $x\in\X$, the transition kernel $Q(x,\cdot)$ admits a probability density function. That is,
$$Q(x,dy)=q(x,y)dy.$$
If $x\mapsto q(x,y)$ is lower semicontinuous for each $y\in\X$, then $Q$ is continuous under the weak star topology, i.e. for any bounded Borel measurable $g:\X\mapsto\R$ and $\{x_n\}_{n\in\N}$ in $\X$ with $x_n\rightarrow x$,
\begin{equation}\label{e9}
\lim_{n\rightarrow\infty}\int_\X g(y)\,Q(x_n,dy)= \int_\X g(y)\,Q(x,dy).
\end{equation}
Indeed, for any Borel measurable $g:\X\mapsto\R$ with $C:=\sup_{x\in\X}|g(x)|< \infty$, Fatou's lemma gives
$$\liminf_{n\rightarrow\infty}\int_\X (C\pm g(y))q(x_n,y)dy\geq\int_\X (C\pm g(y))q(x,y)dy.$$
This implies that
$$\limsup_{n\rightarrow\infty}\int_\X g(y)q(x_n,y)dy\leq\int_\X g(y)q(x,y)dy\leq\liminf_{n\rightarrow\infty}\int_\X g(y)q(x_n,y)dy,$$
and thus \eqref{e9} follows.
\end{remark}


\section{An Example}\label{sec:example}
To illustrate our theoretic results, we focus on a practical stopping problem in this section. Our goal is to characterize explicitly the optimal equilibrium. 

Consider an asset price process $X$ in a discrete-time binomial model. Specifically, let $u>1$ and assume that $X$ takes values in $\X:=\{u^i:\ i=0,\pm 1,\pm 2,\dotso\}$. Suppose that there exsits $p\in (0,1)$ such that
$$p = \P(X_{t+1}/X_t=u)\quad\text{and}\quad  1-p = \P(X_{t+1}/X_t=u^{-1}), \quad \forall t=0,1,2,...$$
We assume additionally that $X$ is a submartingale, which corresponds to the condition $p\ge \frac{1}{u+1}$. 

Let the payoff (or, profit) function be $f(x):= (K-x)^+$ on $\X$, where $K>0$ is a given constant. Also, consider the hyperbolic discount function $\delta(t) := \frac{1}{1+\beta t}$ for $t\in\N\cup\{0\}$, where $\beta>0$ is a given constant. It can be easily verified that \eqref{DI} is satisfied. The objective function \eqref{J} then becomes
\[
J(x,\tau) = \E^x\left[\frac{(K-X_\tau)^+}{1+\beta\tau}\right].
\]
This can be viewed as a real options problem where the management of a company considers an investment plan, which has constant revenue $K$ and stochastic cost evolving as $X$, and would like to decide when to carry it out.

To find the optimal equilibrium $S^* = \bigcap_{S\in\cE} S$, we need to first characterize the collection $\cE$ of all equilibria. Recall from Corollary~\ref{coro:existence of equi} that $\cE$ is nonempty.



\begin{lemma}\label{ll2}
For any $S\in\cE$, $S=(0,y]\cap\X$ for some $y\in(0,K)$.
\end{lemma}

\begin{proof}
Fix $S\in \cE$. We will first prove that $S\subseteq(0,K)\cap \X$. By Remark~\ref{rem:equi nonempty}, $S\neq\emptyset$. Note that $S$ must intersect $(0,K)$. Indeed, if $S\cap (0,K)=\emptyset$, then $J(x,\rho(x,S)) = 0< (K-x)^+ = f(x)$ for all $x\in \X\cap (0,K)$, which contradicts the fact that $S\in\cE$. 
Now, assume to the contrary that $S\cap [K,\infty)\neq\emptyset$. Take $x_0:=\min \{ S \cap[K,\infty)\}$. By definition, $x_0\in S$ and $x_0\ge K$. This, together with $S\cap (0,K)\neq\emptyset$, implies $\P^{x_0}(X_{\rho(x_0,S)}<K)>0$. It follows that 
$$f(x_0)=(K-x_0)^+=0<\E^{x_0}\left[\frac{(K-X_{\rho(x_0,S)})^+}{1+\beta\rho(x_0,S)}\right] = J(x_0,\rho(x_0,S)).$$
This implies that $x_0\notin S$, a contradiction.

Now, to prove the desired result, suppose by contradiction that there exists $S\in\cE$ such that $S\neq (0,y]\cap\X$ for any $y\in(0,K)$. Then, there must exist $x\in\X\setminus S$ such that $(x,\infty)\cap S\neq\emptyset$. Set $w:=\max\{(0,x)\cap S\}$ and $z:=\min\{(x,\infty)\cap S\}$. Then we have $w<x<z<K$, where the last inequality follows from $S\subseteq(0,K)\cap \X$ established above.
Since $X$ is a submartingale, 
$$f(x) = K-x\geq\E^x\left[\frac{K-X_{\rho(x,\{w,z\})}}{1+\beta\rho(x,\{w,z\})}\right]=\E^x\left[\frac{(K-X_{\rho(x,S)})^+}{1+\beta\rho(x,S)}\right] = J(x,\rho(x,S)).$$
This, however, contradicts $x\notin S$ and $S\in\cE$.
\end{proof}

A natural question ensuing Lemma~\ref{ll2} is for which values of $y\in(0, K)$ the set $S=(0,y]\cap\X$ is an equilibrium. To answer this, we need a careful analysis involving random walks. Specifically, let $Y$ be a random walk defined on some probability space $(\bar\Omega,\bar\F,P)$ such that
\[
P(Y_{t+1}-Y_t=1)=p\quad\text{and}\quad P(Y_{t+1}-Y_t=-1)=1-p,\quad \forall t=0,1,2,...
\]
Consider 
\[
\xi:=\inf\{t\geq 0:\ Y_t=0\}
\]
and define 
\begin{equation}\label{alpha's}
\alpha_n:=E^n\left[\frac{1}{1+\beta\xi}\right]\quad \forall n\in\N\qquad \hbox{and}\qquad \alpha':=E^1\left[\frac{1}{1+\beta(\xi+1)}\right],
\end{equation}
where $E^n$ denotes the expectation under $P$ conditional on $Y_0 = n$. Note that $\alpha_n$, $n\in\N$, and $\alpha'$ can all be computed explicitly. For example, 
\begin{equation}\label{alpha's explicit}
\alpha_1=\sum_{k=1}^\infty\frac{\binom{2k-1}{k}p^{k-1}(1-p)^k}{2k-1}\cdot\frac{1}{1+\beta(2k-1)},\quad \alpha'=\sum_{k=1}^\infty\frac{\binom{2k-1}{k}p^{k-1}(1-p)^k}{2k-1}\cdot\frac{1}{1+2\beta k}.
\end{equation}

The next result establishes an upper bound of $y$ for equilibria of the form $(0,y]\cap\X$. 

\begin{lemma}\label{ll3}
If $S=(0,y]\cap\X$ belongs to $\cE$ for some $y\in S\cap(0,K)$, we must have $y\leq U\cdot K$, where
\begin{equation}\label{U}
U:= \frac{1-\frac{1-p}{1+\beta}-p\alpha'}{1-\frac{1-p}{u(1+\beta)}-p\alpha'}.
\end{equation}
\end{lemma}

\begin{proof}
In view of $S\in\cE$ and $y\in S\cap(0,K)$, we have
\begin{align}\label{ll3 eqn}
K-y = f(y) \geq J(y,\rho(y,S)) &= \E^y\left[\frac{(K-X_{\rho(y,S)})^+}{1+\beta\rho(y,S)}\right]\notag\\
&=(1-p)\cdot\frac{K-yu^{-1}}{1+\beta}+p\cdot(K-y)\alpha',
\end{align}
where $\alpha'$ is defined as in \eqref{alpha's}. 
Solving the above inequality for $y$ yields the desired result.
\end{proof}

The next result establishes a lower bound of $y$ for equilibria of the form $(0,y]\cap\X$.

\begin{lemma}\label{ll6}
If $S=(0,y]\cap\X$ belongs to $\cE$ for some $y\in S\cap(0,K)$, we must have $y > L\cdot K$, where
\begin{equation}\label{L}
L:= \frac{1-\alpha_1}{u-\alpha_1}.
\end{equation}
\end{lemma}

\begin{proof}
Since $L<\frac{1}{u}$ by definition, the desired result holds trivially for $y\geq K/u$. In the following, we assume that $y<K/u$. Since $u>1$, we have $y<yu<K$, which particularly implies $yu\notin S$.  It follows that
\begin{equation}\label{1111}
K-yu = f(yu) < J(yu,\rho(yu,S))=\E^{yu}\left[\frac{(K-X_{\rho(yu,S)})^+}{1+\beta\rho_{(yu,S)}}\right],
\end{equation}
where the inequality follows from $yu \notin S$ and $S\in\cE$.  Noting that $\rho(yu,S) = \inf\{t\ge 1 : X^{yu}_t =y\}$, we realize the expectation on the right hand side above can be reduced to an expectation involving the random walk $Y$, i.e.
\begin{equation}\label{2222}
\E^{yu}\left[\frac{(K-X_{\rho(yu,S)})^+}{1+\beta\rho_{(yu,S)}}\right]=  \E^{yu}\left[\frac{K-y}{1+\beta\rho_{(yu,S)}}\right]= E^{1}\left[\frac{K-y}{1+\beta\xi}\right]=\alpha_1(K-y). 
\end{equation}
Combining \eqref{1111} and \eqref{2222} yields $y>\frac{1-\alpha_1}{u-\alpha_1} K$, as desired.
\end{proof}

We intend to show that the upper and lower bounds, $U$ and $L$ in \eqref{U} and \eqref{L}, are in fact {\it sharp}, which induces a complete characterization of $\cE$. To this end, we need the following estimate.

\begin{lemma}\label{ll4}
$\alpha_n\geq(\alpha_1)^n$, for all $n\in\N$.
\end{lemma}

\begin{proof}
Let $\{\F^Y_t\}_{t\ge 0}$ be the natural filtration generated by $Y$. For any $n\in\N$, consider $\sigma:=\inf\{k\geq 0:\ Y_k=n-1\}.$ Then we have
\begin{align*}
\alpha_n&= E^n\left[\frac{1}{1+\beta\xi}\right]= E^n\left[E^n\left[\frac{1}{1+\beta\xi}\Big|\mathcal{F}_\sigma^Y\right]\right]\\
&\geq E^n\left[E^n\left[\frac{1}{1+\beta\sigma}\cdot\frac{1}{1+\beta(\xi-\sigma)}\Big|\mathcal{F}_\sigma^Y\right]\right]=E^n\left[\frac{1}{1+\beta\sigma} E^n\left[\frac{1}{1+\beta(\xi-\sigma)}\Big|\mathcal{F}_\sigma^Y\right]\right]\\
&=\alpha_{n-1}E^n\left[\frac{1}{1+\beta\sigma}\right]=\alpha_{n-1}\cdot\alpha_1.
\end{align*}
The desired result then follows from an induction argument.
\end{proof}

A complete characterization of $\cE$ can now be stated, which in turn gives an explicit formula for the optimal equilibrium. 

\begin{proposition}\label{pp1}
$S\in\cE$ if and only if $S=(0,y]\cap\X$ for some $y\in(L\cdot K,U\cdot K]\cap\X$, where $U$ and $L$ are given in \eqref{U} and \eqref{L}. Consequently, $S^*=(0,y^*]\cap\X$ is the optimal equilibrium, with $y^*:=\min\{(L\cdot K,\infty)\cap\X\}$.
\end{proposition}

\begin{proof}
The sufficiency follows from Lemmas \ref{ll2}, \ref{ll3} and \ref{ll6}. For the necessity, consider $S=(0,y]\cap\X$ for some $y\in(LK,UK]\cap\X$, and we intend to prove that $S\in\cE$. Fix $x\in S$. If $x<y$, since $y\le U\cdot K<K$, by the fact that $X$ is a submartingale and using the optional sampling theorem,
\[
f(x) = K-x>\E^x\left[\frac{K-X_{\rho(x,S)}}{1+\beta\rho(x,S)}\right] = J(x,\rho(x,S)).
\]
On the other hand, if $x=y$, the same argument in \eqref{ll3 eqn} gives
\[
J(x,\rho(x,S)) = (1-p)\cdot\frac{K-yu^{-1}}{1+\beta}+p\cdot(K-y)\alpha'\le K-y = f(y), 
\]
where the inequality is due to $y\le U\cdot K$. Thus, we conclude that $f(x)\ge J(x,\rho(x,S))$ for all $x\in S$. 
Now, fix $x\in\X\setminus S$. 
If $x\geq K$, obviously
$$f(x)=(K-x)^+=0<\E^x\left[\frac{(K-X_{\rho(x,S)})^+}{1+\beta\rho(x,S)}\right]=J(x,\rho(x,S)).$$
If $x<K$, since $x>y$, there must  exist  $n\in\N$ such that $yu^{n-1}< x= yu^{n}$. Since $y>L\cdot K$ and $u>1$, 
\begin{equation}\label{y>}
y> L\cdot K\cdot \frac{1+\alpha_1+(\alpha_1)^2+...+(\alpha_1)^{n-1}}{u^{n-1}+u^{n-2}\alpha_1+u^{n-3}(\alpha_1)^2+...+(\alpha_1)^{n-1}} = \frac{1-(\alpha_1)^n}{u^n-(\alpha_1)^n} K.
\end{equation}
Recall $\alpha_n$ in \eqref{alpha's}, and observe that
\[
J(x,\rho(x,S))=\E^x\left[\frac{K-y}{1+\beta\rho(x,S)}\right] = (K-y) \alpha_n \ge (K-y) (\alpha_1)^n> K-yu^n = K-x=f(x),
\]
where the first inequality follows from Lemma~\ref{ll4} and the second inequality is due to \eqref{y>}. Thus, we conclude that $f(x)< J(x,\rho(x,S))$ for all $x\notin S$. This readily shows that $S\in\cE$.

The second assertion of this proposition follows from either Theorem~\ref{t2} or Theorem \ref{t3}.
\end{proof}

\begin{remark}[Numerical results] Let $\beta=0.2$, $u=1.3$, $p=0.5$ and $K=1$. Then $\X = \{..., (1.3)^{-2}, (1.3)^{-1}, 1,1.3, (1.3)^2,...\}$, and $L=0.5814$, $U=0.7737$ (by \eqref{L}, \eqref{U}, and \eqref{alpha's explicit}). Since $(L\cdot K, R\cdot K] \cap \X = (0.5814, 0.7737]\cap \X = \{(1.3)^{-2}, (1.3)^{-1}\}$, we conclude from Proposition~\ref{pp1} that 
\[
\cE = \{(0,(1.3)^{-2}]\cap\X, (0,(1.3)^{-1}]\cap\X\} = \{(0,0.5917]\cap\X, (0,0.7692]\cap\X\},
\]
and $(0,0.5917]\cap\X$ is the optimal equilibrium.
\end{remark}


\section{Conclusion}\label{sec:conclusion}
For a time-inconsistent problem in discrete time, questions (a) and (b) proposed in Section~\ref{sec:introduction} are nontrivial, when the time horizon is infinite. In this paper, we focus on an infinite-horizon stopping problem under non-exponential discounting, and develop the iterative approach to resolve both (a) and (b). 




The iterative approach is motivated by the continuous-time framework in Huang and Nguyen-Huu \cite{HN17}. The current discrete-time setup, however, gives rise to fundamental differences. First, when rephrased using the notation in the present paper, \cite[Proposition 3.2]{HN17} states that $\bigcup_{n\in\N} \Theta^n(S)$ is an equilibrium, as long as $S\subseteq \Theta(S)$ holds in the first iteration. While the condition $S\subseteq \Theta(S)$ is easily satisfied in continuous time (as explained in \cite[Remark 3.5]{HN17}), it is not true in general in discrete time. In response to this, we develop the fixed-point iteration in the opposite way: Theorem~\ref{t1} shows that $ \bigcap_{n\in\N} \Theta^n(S)$ is an equilibrium, as long as $\Theta(S)\subseteq S$ holds in the first iteration. This in particular leads to the existence of an optimal equilibrium (Theorems~\ref{t2} and \ref{t3}), a result that has not been established in continuous time. Note that the proofs in the present paper hinge on the discrete-time structure, and cannot be extended to continuous time in an obvious way. It is therefore of interest to investigate, as a future research project, whether an optimal equilibrium exists in continuous time.


\begin{thebibliography}{10}

\bibitem{Ainslie-book-92}
{\sc G.~Ainslie}, {\em Picoeconomics}, Cambridge University Press, Cambridge,
  UK, 1992.

\bibitem{Asheim97}
{\sc G.~Asheim}, {\em Individual and collective time-consistency}, Review of
  Economic Studies, 64 (1997), pp.~427--443.

\bibitem{BS12}
{\sc E.~Bayraktar and M.~S\^\i~rbu}, {\em Stochastic {P}erron's method and
  verification without smoothness using viscosity comparison: the linear case},
  Proc. Amer. Math. Soc., 140 (2012), pp.~3645--3654.

\bibitem{BKM17}
{\sc T.~Bj{\"o}rk, M.~Khapko, and A.~Murgoci}, {\em On time-inconsistent
  stochastic control in continuous time}, Finance and Stochastics, 21 (2017),
  pp.~331--360.

\bibitem{BMZ14}
{\sc T.~Bj{\"o}rk, A.~Murgoci, and X.~Y. Zhou}, {\em Mean-variance portfolio
  optimization with state-dependent risk aversion}, Math. Finance, 24 (2014),
  pp.~1--24.

\bibitem{EL06}
{\sc I.~Ekeland and A.~Lazrak}, {\em Being serious about non-commitment:
  subgame perfect equilibrium in continuous time}, tech. rep., University of
  British Columbia, 2006.
\newblock Available at http://arxiv.org/abs/math/0604264.

\bibitem{EMP12}
{\sc I.~Ekeland, O.~Mbodji, and T.~A. Pirvu}, {\em Time-consistent portfolio
  management}, SIAM J. Financial Math., 3 (2012), pp.~1--32.

\bibitem{EP08}
{\sc I.~Ekeland and T.~A. Pirvu}, {\em Investment and consumption without
  commitment}, Math. Financ. Econ., 2 (2008), pp.~57--86.

\bibitem{GW07}
{\sc S.~R. Grenadier and N.~Wang}, {\em Investment under uncertainty and
  time-inconsistent preferences}, Journal of Financial Economics, 84 (2007),
  pp.~2--39.

\bibitem{Hayek36}
{\sc F.~Hayek}, {\em Utility analysis and interest}, The Economic Journal, 46
  (1936), pp.~44--60.

\bibitem{HJZ12}
{\sc Y.~Hu, H.~Jin, and X.~Y. Zhou}, {\em Time-inconsistent stochastic
  linear-quadratic control}, SIAM J. Control Optim., 50 (2012), pp.~1548--1572.

\bibitem{HN17}
{\sc Y.-J. Huang and A.~Nguyen-Huu}, {\em Time-consistent stopping under
  decreasing impatience}, Finance and Stochastics, 22 (2018), pp.~69--95.

\bibitem{HNZ17}
{\sc Y.-J. Huang, A.~Nguyen-Huu, and X.~Y. Zhou}, {\em General stopping behaviors of
  naive and non-committed sophisticated agents, with application to probability distortion},
  preprint,  (2018).
\newblock Available at https://arxiv.org/abs/1709.03535.

\bibitem{karp2007non}
{\sc L.~Karp}, {\em Non-constant discounting in continuous time}, Journal of
  Economic Theory, 132 (2007), pp.~557--568.

\bibitem{Kocher96}
{\sc N.~Kocherlakota}, {\em Reconsideration-proofness: A refinement for
  infinite horizon time inconsistency}, Games and Economic Behavior, 15 (1996),
  pp.~33--54.

\bibitem{Laibson97}
{\sc D.~Laibson}, {\em Golden eggs and hyperbolic discounting}, Q. J. Econ, 112
  (1997), pp.~443--477.

\bibitem{LP92}
{\sc G.~Loewenstein and D.~Prelec}, {\em Anomalies in intertemporal choice:
  evidence and an interpretation}, Q. J. Econ., 57 (1992), pp.~573--598.

\bibitem{LT89}
{\sc G.~Loewenstein and R.~Thaler}, {\em Anomalies: Intertemporal choice},
  Journal of Economic Perspectives, 3 (1989), pp.~181--193.

\bibitem{NP05}
{\sc A.~Molavi~Vasséi}, {\em Recursive utility, increasing impatience and
  capital deepening: F.{A}. hayek's 'utility analysis and interest'}, The
  European Journal of the History of Economic Thought, 22 (2015),
  pp.~1000--1041.

\bibitem{Noor2009a}
{\sc J.~Noor}, {\em Decreasing impatience and the magnitude effect jointly
  contradict exponential discounting}, Journal of Economics Theory, 144 (2009),
  pp.~869--875.

\bibitem{PY73}
{\sc B.~Peleg and M.~E. Yaari}, {\em On the existence of a consistent course of
  action when tastes are changing}, Review of Economic Studies, 40 (1973),
  pp.~391--401.

\bibitem{PP68}
{\sc E.~Phelps and R.~A. Pollak}, {\em On second-best national saving and
  game-equilibrium growth}, Review of Economic Studies, 35 (1968),
  pp.~185--199.

\bibitem{Pollak68}
{\sc R.~A. Pollak}, {\em Consistent planning}, The Review of Economic Studies,
  (1968), pp.~201--208.

\bibitem{Prelec04}
{\sc D.~Prelec}, {\em Decreasing impatience: A criterion for non-stationary
  time preference and ``hyperbolic'' discounting}, Scand. J. of Econ., 106
  (2004), pp.~511--532.

\bibitem{Stahl93}
{\sc D.~Stahl}, {\em Evolution of smart-n players}, Games and Economic
  Behavior, 5 (1993), pp.~604--617.

\bibitem{SW94}
{\sc D.~Stahl and P.~Wilson}, {\em Experimental evidence on players' models of
  other players}, Journal of Economic Behavior and Organization, 25 (1994),
  pp.~309--327.

\bibitem{Strotz55}
{\sc R.~H. Strotz}, {\em Myopia and inconsistency in dynamic utility
  maximization}, The Review of Economic Studies, 23 (1955), pp.~165--180.

\bibitem{Thaler81}
{\sc R.~Thaler}, {\em Some empirical evidence on dynamic inconsistency}, Econ.
  Lett., 8 (1981), pp.~201--207.

\bibitem{Yong12}
{\sc J.~Yong}, {\em Time-inconsistent optimal control problems and the
  equilibrium \textsc{HJB} equation}, Mathematical Control and Related Fields,
  3 (2012), pp.~271--329.

\end{thebibliography}

\end{document}